\documentclass[a4paper,11pt,reqno]{amsart}
\usepackage{amsmath}
\usepackage{amsthm,enumerate}

\usepackage{graphicx}
\usepackage{comment}
\usepackage{amssymb}

\usepackage{appendix}

\usepackage[italian,english]{babel}

\usepackage[utf8x]{inputenc}
\usepackage{fancyhdr}

\usepackage{calc}
\usepackage{url}

\usepackage[text={6.3in,8.6in},centering]{geometry}

\usepackage{srcltx, inputenc}

\usepackage[T1]{fontenc}

\usepackage[usenames,dvipsnames]{color}

\makeatletter
\def\cleardoublepage{\clearpage\if@twoside \ifodd\c@page\else%
         \hbox{}%
     \thispagestyle{empty}
     \newpage%
     \if@twocolumn\hbox{}\newpage\fi\fi\fi}
\makeatother

\theoremstyle{plain}
\newtheorem{theorem}{Theorem}[section]

\newtheorem{definition}[theorem]{Definition}

\newtheorem{lemma}[theorem]{Lemma}

\newtheorem{proposition}[theorem]{Proposition}
\newtheorem{remark}[theorem]{Remark}

\numberwithin{equation}{section}
\theoremstyle{definition}

\newcommand{\R}{\ensuremath{\mathbb{R}}}

\begin{document}

\title[]{On a semilinear parabolic equation\\ with time-dependent source term \\ on infinite graphs}

\author{Fabio Punzo}
\address{\hbox{\parbox{5.7in}{\medskip\noindent{Dipartimento di Matematica,\\
Politecnico di Milano,\\
   Piazza Leonardo da Vinci 32, 20133 Milano, Italy. \\[3pt]
        \em{E-mail address: }{\tt
          fabio.punzo@polimi.it}}}}}

          \author{Alessandro Sacco}
\address{\hbox{\parbox{5.7in}{\medskip\noindent{Dipartimento di Matematica,\\
Politecnico di Milano,\\
   Piazza Leonardo da Vinci 32, 20133 Milano, Italy. \\[3pt]
        \em{E-mail address: }{\tt
alessandro3.sacco@mail.polimi.it}}}}}



\maketitle

\maketitle

\begin{abstract} We are concerned with semilinear parabolic equations, with a time-dependent source term of the form $h(t)u^q$ with $q>1$, posed on an infinite graph. We assume that the bottom of the $L^2$-spectrum of the Laplacian on the graph, denoted by $\lambda_1(G)$, is positive. In dependence of $q, h(t)$ and $\lambda_1(G)$, we show global in time existence or finite time blow-up of solutions.
\end{abstract}

\bigskip
\bigskip

\section{Introduction}
We investigate the global existence and finite-time blow-up of solutions to parabolic Cauchy problems of the form:
\begin{equation}\label{problema}
\begin{cases}
u_t= \Delta u +h(t)u^q & \text{in } G\times (0,T) \\
u =u_0&\text{in } G\times {0},
\end{cases}
\end{equation} where $(G, \omega, \mu)$ is an {\it infinite weighted} graph with {\it edge-weight} $\omega$ and {\it vertex measure} $\mu$, $\Delta$ denotes the Laplace operator on $G$, $T \in (0,\infty]$, and $u_0: G \to \mathbb{R}$ is a given nonnegative initial datum.

The existence and nonexistence of global solutions for problem \eqref{problema} in the Euclidean setting, particularly when $h \equiv 1$, have been extensively studied in the literature (see, e.g., \cite{F}, \cite{H}, \cite{KST}, \cite{BB}, \cite{DL}, \cite{Levine}, and references therein). When the underlying domain is a Riemannian manifold instead of $\mathbb{R}^N$, different behaviors emerge, as shown in \cite{BPT, MeGP1, MeGP2, MeGP3, MeGP4, GMPu, Sun1, MaMoPu, Punzo, Pu22, Zhang}. Notably, the specific case of problem \eqref{problema} on the hyperbolic space was analyzed in \cite{BPT}.

In recent years, there has been growing interest in the study of elliptic and parabolic equations on graphs. Important contributions in this direction include the monographs \cite{Grig2, KLW, Mu2} and the works \cite{AS1, BMP, HJ, HK, MP2, MoPuSo} for elliptic equations, as well as \cite{BCG, BP, F, GT, HMu, Huang, HuangKS, LSZ, LW2, GMe, MoPuSo2, Mu2, PuSv, PuTe, PuTe2, PuTe3, Wu} for parabolic ones. In particular, several results concerning the blow-up and global existence of solutions to problem \eqref{problema},  with $h \equiv 1$ and a general source term $f(u)$, have been established in \cite{GMeP5, LSZ, LW2, LWu}.

In this paper, we extend some of the results obtained in \cite{BPT} for the hyperbolic space to the setting of graphs. As a model case, we consider $h(t) = e^{\alpha t}$. Let $\lambda_1(G)$ denote the bottom of the $L^2$-spectrum of $\Delta$ on $G$. We prove that any nontrivial solution to problem \eqref{problema} blows up in finite time whenever $\alpha > (q-1)\lambda_1(G)$. Conversely, if $\alpha < (q-1)\lambda_1(G)$, a global solution exists provided that the initial datum $u_0$ is sufficiently small. The critical case $\alpha = (q-1)\lambda_1(G)$ remains an open question. For the hyperbolic space, this case has been investigated in \cite{BPT, WY} using precise heat semigroup estimates and properties of solutions to an associated elliptic equation.

Furthermore, when $h \equiv 1$, we show that for any $q > 1$, problem \eqref{problema} admits a global solution if $u_0$ is small enough, which aligns with the results in \cite{GMeP5}.

It is worth noting that our proof of the blow-up result follows the same strategy as in \cite{BPT} (see also \cite{GMeP5}), suitably adapted to our setting. However, the proofs of local and global existence of solutions differ substantially. In \cite{BPT}, existence is established via the method of sub- and supersolutions, which we are unable to construct in our framework. Instead, our approach relies on a fixed-point argument. A similar strategy was employed in \cite{GMeP5}, but due to the presence of the function $h(t)$, we utilize a different metric space to apply the contraction principle, therefore, some different estimates are in order. Notably, our method is quite general and could also provide an alternative proof for the results in \cite{BPT, PuTe}, where the same equation is studied on hyperbolic spaces and {\it metric} graphs.

The structure of the paper is as follows. In Section \ref{prel}, we recall fundamental concepts related to graphs and the heat semigroup on a graph. The notion of solution we consider, along with our main results, is presented in Section \ref{statements}. The finite-time blow-up result is established in Section \ref{blowup}, while local and global existence are addressed in Sections \ref{local-existence} and \ref{existence}, respectively.

\section{Mathematical background}\label{prel}

Let $G$ be a countable infinite set, and
let $\mu:G\to (0,+\infty)$ be a measure on $G$
satisfying $\mu(\{x\}) <+\infty$ for every $x\in G$ (so that $(G,\mu)$ becomes a measure space). Furthermore, let
\begin{equation*}
\omega:G\times G\to [0,+\infty)
\end{equation*}
be a symmetric, with zero diagonal and finite sum function, i.e.
\begin{equation}\label{omega}
\begin{aligned}
&\text{(i)}\,\, \omega(x,y)=\omega(y,x)\quad &\text{for all}\,\,\, (x,y)\in G\times G;\\
&\text{(ii)}\,\, \omega(x,x)=0 \quad\quad\quad\,\, &\text{for all}\,\,\, x\in G;\\
&\text{(iii)}\,\, \displaystyle \sum_{y\in G} \omega(x,y)<\infty \quad &\text{for all}\,\,\, x\in G\,.
\end{aligned}
\end{equation}
Thus, we define  \textit{weighted graph} the triplet $(G,\omega,\mu)$, where $\omega$ and $\mu$ are the so called \textit{edge weight} and \textit{node measure}, respectively. Observe that assumption $(ii)$ corresponds to require that $G$ has no loops.
\smallskip

\noindent Let $x,y$ be two points in $G$; we say that
\begin{itemize}
\item $x$ is {\it connected} to $y$, and we write $x\sim y$, whenever $\omega(x,y)>0$;
\item the couple $(x,y)$ is an {\it edge} of the graph and the vertices $x,y$ are called the {\it endpoints} of the edge whenever $x\sim y$;
\item a collection of vertices $ \{x_k\}_{k=0}^n\subset G$ is a {\it path} if $x_k\sim x_{k+1}$ for all $k=0, \ldots, n-1.$
\end{itemize}

\noindent We say that the weighted graph $(G,\omega,\mu)$ is
\begin{itemize}
\item[(i)] {\em locally finite} if each vertex $x\in G$ has only finitely many $y\in G$ such that $x\sim y$;
\item[(ii)] {\em connected} if, for any two distinct vertices $x,y\in G$ there exists a path joining $x$ to $y$;
\item[(iii)] {\em undirected} if its edges do not have an orientation.
\end{itemize}
We shall always assume in the rest of the paper that the previous properties are fulfilled.

\medskip

Let $\mathfrak F$ denote the set of all functions $f: G\to \mathbb R$. For any $f\in \mathfrak F$ and for all $x,y\in G$, let us give the following
\begin{definition}\label{def1}
Let $(G, \omega,\mu)$ be a weighted graph. For any $f\in \mathfrak F$:
\begin{itemize}
\item the {\em difference operator} is
\begin{equation}\label{e2f}
\nabla_{xy} f:= f(y)-f(x)\,;
\end{equation}
\item the {\em (weighted) Laplace operator} on $(G, \omega, \mu)$ is
\begin{equation*}
\Delta f(x):=\frac{1}{\mu(x)}\sum_{y\in G}[f(y)-f(x)]\omega(x,y)\quad \text{ for all }\, x\in G\,.
\end{equation*}
\end{itemize}
\end{definition}
\noindent Clearly,
\[\Delta f(x)=\frac 1{\mu(x)}\sum_{y\in G}(\nabla_{xy} f)\omega(x,y)\quad \text{ for all } x\in G\,.\]

We set
$$\ell^\infty(G):=\{u\in \mathfrak F\,:\, \sup_{x\in G}|u(x)| <+\infty\}.
$$

\subsection{The heat semigroup on $G$}\label{heats}
We need to recall some preliminaries concerning the heat semigroup on $G$. Let $(G,\omega,\mu)$ be a weighted infinite graph.
Let $\{e^{t\Delta}\}_{t\ge 0}$ be the heat semigroup of $G$. It admits a (minimal) \textit{heat kernel}, namely a function $p:G\times G\times (0,+\infty)\to\R$, $p>0$ in $G\times G\times (0,+\infty)$ such that
\begin{equation*}
(e^{t\Delta} u_0)(x)=\sum_{y\in G}p(x,y,t)\,u_0(y)\,\mu(y), \quad x\in G,\,\, t>0,
\end{equation*}
for any $u_0\in \ell^\infty(G)$. It is well known that
\begin{equation}\label{eq23}
\sum_{y\in G} p(x,y,t)\,\mu(y)\,\le \,1,\quad \text{for all}\,\,\, x\in G,\,\, t>0.
\end{equation}
We say that a graph $G$ is {\it stochastically} complete if the following condition holds:
\begin{equation*}
\sum_{y\in G} p(x,y,t) \,\mu(y)\,= \,1,\quad \text{for all}\,\,\, x\in G,\,\, t>0.
\end{equation*}
The main properties of $p(x,y,t)$ are stated below.
\begin{proposition}\label{prop1}
Let $(G,\omega,\mu)$ be a weighted infinite graph. Then the corresponding heat kernel satisfies the following properties:
\begin{itemize}
\item symmetry: $p(x,y,t)\equiv p(y,x,t)$ for all $x,y\in G$ and $t>0$.
\item $p(x,y,t)\ge 0$ for all $x,y\in G$ and $t>0$ and
\begin{equation}\label{eq23}
\sum_{y\in G} p(x,y,t) \,\mu(y)\,\le \,1,\quad \text{for all}\,\,\, x\in G,\,\, t>0.
\end{equation}
\item semigroup identity: for all $x,y\in G$ and $t,s>0$,
\begin{equation}\label{eq24}
p(x,y,t+s)=\sum_{z\in G} p(x,z,t)\, p(z,y,s) \,\mu(z).
\end{equation}
\end{itemize}
\end{proposition}

The next result is contained in \cite[Proposition 4.5]{KLW}.
\begin{proposition}\label{prop2}
Let $(G,\omega,\mu)$ be a weighted infinite graph. For all vertices $x$ and $y$
\begin{equation}\label{eq26}
\lim_{t\to+\infty} \frac{\log p(x,y,t)}{t}\,=\, -\lambda_1(G)
\end{equation}
where $\lambda_1(G)$ is the infimum of the spectrum of the operator $-\Delta$.
\end{proposition}

Moreover, by combining together \cite[Theorems 2.1, 2.2]{F}, it is also possible to state the following result:
\begin{proposition}\label{prop2bis}
Let $(G,\omega,\mu)$ be a weighted infinite graph. For any $\underline t>0$ there exists $\underline C>0$ such that
\begin{equation}\label{e301}
p(x,y,t)\leq \underline C\, e^{-\lambda_1(G) t} \quad \text{for all}\,\,\,x,y\in G, t\geq \underline t\,,
\end{equation}
where $\lambda_1(G)$ is the infimum of the spectrum of the operator $-\Delta$.

\end{proposition}

Finally, let us recall the \textit{positivity improving} property of the heat semigroup (see, e.g., \cite[Theorem 1.26]{KLW}).
\begin{proposition}
    \label{prop:positivity-improving}
    Let $(G,\omega,\mu)$ be a weighted infinite graph, then
  \begin{equation*}
       \sum_{y\in G}\,p(x, y, t)u_0(y)\mu(y)>0\quad\text{for all } x\in G,\,t>0\,,
    \end{equation*}
   for any $u_0\in\ell^\infty$ such that $u_0\geq 0$ in $G$ and $u_0\not\equiv 0$.
\end{proposition}

\section{Statements of main results}\label{statements}
Solutions to problem \eqref{problema} are meant in the {\it mild} sense, according to the next definition. Here and hereafter $p:G\times G\times (0, +\infty)\to (0, +\infty)$ stands again for the heat kernel on $G$ discussed in the previous section.
\begin{definition}\label{defsol}
A function $u:G\times(0,\tau)\to\R$, $u\in L^\infty((0,\tau),\ell^\infty(G))$ is a \textit{mild} solution of problem \eqref{problema} if
\begin{equation}\label{solmild}
u(x,t)=\sum_{y\in G}p(x,y,t)u_0(y)+\int_0^t\sum_{y\in G}\,p(x,y,t-s)\,h(s)\,u(y,s)^q\,ds\,,
\end{equation}
for every $x\in G$ and $t\in(0,\tau)$. If $u$ is a mild solution of \eqref{problema} in $G\times(0,+\infty)$ then we call it a global solution.
\end{definition}

We say that a solution {\em blows up in finite time}, whenever there exists $\tau>0$ such that
$$\lim_{t\to \tau^-}\|u(t)\|_{\ell^\infty(G)}=+\infty\,.$$

\medskip

Our first result is concerned with nonexistence of global solutions.

\begin{theorem}\label{teo1}
Let $(G,\omega,\mu)$ be a weighted, stochastically complete, infinite graph with $\lambda_1(G)>0$. Let $u_0\in\ell^\infty(G)$, $u_0\ge 0$, $u_0\not\equiv 0$ in $G$. Let $h$ be a non-negative, locally integrable function in $[0, +\infty)$ and define $H:[0,+\infty)\rightarrow[0, +\infty)$ as
\begin{equation*}
    H(t):=\int_0^t\,h(s)\,ds\,.
\end{equation*}
Suppose that, for some $\epsilon\in(0, \lambda_1(G))$,
\begin{equation}\label{eq:limite-non-esistenza}
\lim_{t\rightarrow+\infty}\,\frac{[H(t)]^{\frac{1}{q-1}}}{e^{[\lambda_1(G)+\epsilon]t}}
=+\infty\,.
\end{equation}
Then any solution to problem \eqref{problema} blows up in finite time.
\end{theorem}

Concerning the local existence of solutions for the problem \eqref{problema}, we obtain the following result.

\begin{theorem}\label{locexi}
    Let $(G,\omega,\mu)$ be an infinite weighted graph. Let $h:[0, +\infty)\to [0, +\infty)$ be a locally integrable function, and let $q>1$. Finally, let $u_0\in\mathfrak{F}$ be any nonnegative and bounded function. Then there exists a solution $u\in L^{\infty}((0, T); \ell^{\infty}(G))$ to problem \eqref{problema}, provided that $T>0$ is sufficiently small.
    \label{theo:local-existence}
\end{theorem}

Note that in Theorem \ref{locexi}, the assumption $\lambda_1(G)>0$ has not been used. Furthermore, we prove the following result about global existence.

\begin{theorem}\label{teo2}
Let $(G,\omega,\mu)$ be an infinite weighted graph with $\lambda_1(G)>0$. Let $h:[0, +\infty)\to [0, +\infty)$ be a locally integrable function, and let $q>1$. Assume that
\begin{equation}\label{eq306}
    \tilde{H}:=\int_0^{+\infty}\,h(t)\,e^{-\lambda_1(G)(q-1)t}\,dt<+\infty\,.
\end{equation}
Moreover, assume that $u_0\in\ell^\infty(G)$, and
\begin{equation}\label{eq12a}
0\leq u_0(x)\leq\epsilon\,p(x, y_0,\gamma)\qquad\text{for all }x\in G\,,
\end{equation}
for some $\gamma>0,\,y_0\in G$ and for a sufficiently small $\epsilon>0.$
Then there exists a global solution $u:G\times[0, +\infty)\rightarrow[0, +\infty)$ to problem \eqref{problema}.
In addition, for some $M>0$,
\begin{equation}\label{eq305}
0\le u(x,t)\le M p(x,y_0,t+\gamma) \quad \text{ for all } x\in G, t\geq 0 .
\end{equation}

\end{theorem}

\begin{remark}
Let
\[h(t)=e^{\alpha t}, \quad t\geq 0\,.\]
From Theorems \ref{teo1} and \ref{teo2} we deduce that if $\alpha>\lambda_1(G)(q-1)$, then any solution to problem \eqref{problema} blows up in finite time (provided that $u_0\not \equiv 0)$.
On the other hand, if $\alpha<\lambda_1(G)(q-1)$, then a global in time solution exists, provided that $u_0$ is small enough.

It remains to understand what happens in the critical case $\alpha=\lambda_1(q-1).$
\end{remark}

\begin{remark}
Let $h(t)\equiv 1$. From Theorem \ref{teo2} we have that for any $q>1$ there exists a global solution, provided that $u_0$ is small enough. This is in accordance with the results in \cite{GMeP5}.
\end{remark}

\section{Finite time blow-up for any initial datum}\label{blowup}

We deal here with the blowup analysis. We start with some technical tools.

\subsection{Two key estimates}

Let us first prove a preliminary lemma.

\begin{lemma}\label{lemma1}
Let $(G,\omega,\mu)$ be a weighted infinite graph with $\lambda_1(G)>0$. Suppose that $u_0\in \mathcal F, u_0\geq 0$, $u_0(x_0)>0$ for some $x_0\in G.$
Let $\varepsilon\in(0,\lambda_1(G))$. 
Then there exists $t_0=t_0(x_0,\varepsilon)>0$ such that
\begin{equation}\label{eq31}
(e^{t\Delta}u_0)(x_0) \ge \, C_1 e^{-[\lambda_1(G)+\varepsilon]t}\,, \quad \text{for any}\,\,\, t>t_0\,,
\end{equation}
where $C_1:=u_0(x_0) \mu(x_0)\,.$
\end{lemma}

\begin{proof}
Let $x_0 \in G$. From \eqref{eq26} if follows that there exists $t_0>0$ such that
$$
p(x_0,x_0,t)\ge e^{-[\lambda_1(G)+\varepsilon]t} \quad \text{for every}\, t>t_0\,.
$$
Hence
$$
\begin{aligned}
(e^{t\Delta}u_0)(x) &=\sum_{y\in G} p(x_0,y,t)u_0(y)\mu(y)\\
&\ge p(x_0,x_0,t) u_0(x_0)\, \mu(x_0) \\
&\ge e^{-[\lambda_1(G)+\varepsilon]t}\, u_0(x_0) \mu(x_0).
\end{aligned}
$$
Consequently, we obtain \eqref{eq31} with $C_1:=u_0(x_0) \mu(x_0)>0\,.$
\end{proof}

Let $u$ be a solution of equation \eqref{problema}. Then, for any $x\in G$ and for any $T>0$, we define
\begin{equation}\label{eq32}
\Phi_x^T(t)\equiv \Phi_x(t):=\sum_{z\in G} p(x,z,T-t)\, u(z,t)\, \mu(z)\,\quad \text{ for any }\,\, t\in [0, T]\,.
\end{equation}
Observe that
\begin{equation}\label{eq33}
\Phi_x(0)=\sum_{z\in G} p(x,z,T) \,u_0(z)\,\mu(z)\,= (e^{T\Delta}u_0)(x), \,\, x\in G\,.
\end{equation}
We now state the following lemma.

\begin{lemma}\label{lemma2}
Let $(G,\omega,\mu),\, q,\, h,\, u_0,\, H$ be as in Theorem \ref{teo1}. Let $x\in G$ and $\Phi_x^T(t)$ be as in \eqref{eq32}. Then
\begin{equation}\label{eq34}
[\Phi_x^T(0)]^{q-1}\,\,\leq\,\, \frac{1}{(q-1)H(T)}\,.
\end{equation}
\end{lemma}

\begin{proof}
Let $u$ be a solution to problem \eqref{problema}, hence
\begin{equation}\label{eq30f}
u(z,t)=\sum_{y\in G} p(z,y,t) u_0(y)\,\mu(y) + \int_0^t\,\sum_{y\in G}\,p(z,y,t-s)\,h(s)\,u(y,s)^q\,\mu(y)\, ds\,.
\end{equation}

We multiply the equality by $p(x,z,T-t)$ and sum over $z\in G$. Therefore, we get
\begin{equation}\label{eq35}
\begin{aligned}
\sum_{z\in G} p(x,z,T-t)\,u(z,t)\,\mu(z) &=\sum_{z\in G}\sum_{y\in G} p(z,y,t)\,u_0(y)\,p(x,z,T-t)\,\mu(z)\mu(y)\\
&+ \int_0^t\sum_{z\in G}\sum_{y\in G} p(z,y,t-s)\,p(x,z,T-t)\,\mu(z)\,h(s)\,u(y,s)^q\mu(y)\,ds.
\end{aligned}
\end{equation}
Now, due to \eqref{eq32}, for all $t\in (0, T),$ equality \eqref{eq35} reads
$$
\begin{aligned}
\Phi_x(t)&=\sum_{z\in G}\sum_{y\in G} p(z,y,t)\,u_0(y)\,p(x,z,T-t)\,\mu(z)\mu(y) \\
&+ \int_0^t\sum_{z\in G}\sum_{y\in G} p(z,y,t-s)\,p(x,z,T-t)\,\mu(z)\,h(s)\,u(y,s)^q\,\mu(y)\,ds.
\end{aligned}
$$
By \eqref{eq33}, for all $t\in (0, T),$
\begin{equation*}
\begin{aligned}
\Phi_x(t)&= \sum_{y\in G} p(x,y,T)\,u_0(y)\mu(y)+\int_0^t\sum_{y\in G}\,p(x,y,T-s)\,h(s)\,u(y,s)^q\,\mu(y)\,ds \\
&= \Phi_x(0)+\int_0^t\sum_{y\in G}\,p(x,y,T-s)\,h(s)\,u(y,s)^q\,\mu(y)\,ds\,.
\end{aligned}
\end{equation*}
Differentiating at both sides, we get
\begin{equation*}
    \Phi_x'(t)=\sum_{y\in G}\,p(x, y, T-t)\,u(y,t)^q\,\mu(y)\,h(t)\,.
\end{equation*}
Clearly, the map $s\mapsto s^q$ is convex, therefore the Jensen inequality may be used on the
 right hand side, yielding
\begin{equation*}
    \Phi_x'(t)\geq h(t)\,\left(\sum_{y_\in G}\,p(x, y, T-t)\,u(y, t)\,\mu(y)\right)^q=h(t)\,[\Phi_x(t)]^q\,.
\end{equation*}
From Proposition \ref{prop:positivity-improving}, we have that $\Phi_x(t)>0$ for all $t\in[0, T]$, then the inequality can be divided by $[\Phi_x(t)]^q$ and integrated in $[0, T]$. The following estimate is the result of the integration
\begin{equation*}
    (q-1)H(T)\leq\frac{1}{[\Phi_x(0)]^{q-1}}-\frac{1}{[\Phi_x(T)]^{q-1}}\,.
\end{equation*}
 The thesis follows immediately.
\end{proof}

\subsection{Proof of Theorem \ref{teo1}}
\begin{proof}[Proof of Theorem \ref{teo1}]
Suppose, by contradiction, that $u$ is a global solution of problem \eqref{problema} and let $\epsilon$ be any value in $(0, \lambda_1(G))$.  Let $x_0\in G$. By Lemma \ref{lemma1} and Lemma \ref{lemma2},
$$
C_1 \, e^{-[\lambda_1(G)+\varepsilon]\,T}\,\,\le\,\,(e^{T\Delta}u_0)(x)= \Phi_x(0)\,\,\le\,\, \left(\frac{1}{q-1}\right)^{\frac{1}{q-1}}\,H(T)^{-\frac{1}{q-1}}, \quad \text{for any}\,\,\, T>t_0,
$$
where $t_0>0$, $C_1>0$ are given in Lemma \ref{lemma1}. Hence, if $u$ exists globally in time, we would have
\begin{equation}\label{eq326}
\frac{H(T)^{\frac{1}{q-1}}}{e^{[\epsilon+\lambda_1(G)]T}}\le\,\, \frac{1}{C_1}\left(\frac{1}{q-1}\right)^{\frac{1}{q-1}} \quad \text{for any }\, T>t_0.
\end{equation}
Nonetheless, due to \eqref{eq:limite-non-esistenza}, the left hand side of \eqref{eq326} tends to $+\infty$ as $T\to \infty$.
Thus, we have a contradiction. Hence the thesis follows.
\end{proof}

\section{Local existence}\label{local-existence}
The proof for the local existence of a solution for problem \eqref{problema} relies on an application of the Banach-Caccioppoli theorem in the complete metric space:
\[\mathcal B_T:=\{u\in \mathfrak F\,:\, 0\leq u(x,t) \leq M\quad \forall\, x\in X, t\in (0, T)\}\qquad (T>0),\]
endowed with the distance
\[d_\infty(u, v):=\sup_{G\times(0, T)}|u-v|\,.\]

We define the map $\Psi:\mathcal{B}_T\rightarrow\mathcal{B}_T$ as
\begin{equation}\label{eq:psi-local}
(\Psi u)(x,t):=\sum_{z\in G}p(x,z,t)u_0(z)\mu(z)+\int_0^t\sum_{z\in G}p(x,z,t-s)\,h(s)u(z,s)^q\,\mu(z)\,ds\,.
\end{equation}

We first prove two key estimates that ensure that the map $\Psi$ is a contraction map in the metric space $(\mathcal{B}_T,d_\infty)$.

\begin{lemma}
    Let $u_0\in\ell^\infty(G), u_0\geq 0$. Suppose that $M>\|u_0\|_\infty.$
    Then, for any $u\in\mathcal{B}_T$,
    \begin{equation*}
        \Psi(u)\in\mathcal{B}_T\,,
    \end{equation*}
    \label{lemma:psi-lemma-1-local}
provided that $T>0$ is small enough.
\end{lemma}
\begin{proof}
   From \eqref{eq:psi-local}, \eqref{eq23} and the very definition of $\mathcal B_T$, we get
    \begin{align*}
        0\leq (\Psi u)(x, t) &= \sum_{y\in X}\,p(x, y, t)u_0(y)\mu(y)\,+ \\
        &+\int_0^t\,\sum_{y\in X}\,p(x, y, t-s) h(s)u(y, s)^{q}\mu(y)\,ds \nonumber \\
        &\leq \|u_0\|_\infty + M^q \int_0^T h(s) ds\,,\nonumber
    \end{align*}
    for any $x\in G$, $t\in (0,T)$. Since $M>\|u_0\|_\infty,$ we can select $T>0$ so small that
    \begin{equation}\label{eq330}
    \|u_0\|_\infty + M^q \int_0^T h(s) ds \leq M\,.
    \end{equation}
    Thus $\Psi(u)\in\mathcal{B}_T$.
\end{proof}

\begin{proposition}\label{lemma:psi-lemma-2-local}
Let the assumptions of Theorem \ref{theo:local-existence} be satisfied, and suppose that $M>\|u_0\|_\infty.$
Then, for any $u,v\in\mathcal{B}_T$,
\begin{equation}\label{eq55aa}
d_\infty(\Psi(u), \Psi(v))\leq\, q\,M^{q-1}\int_0^{T}\,h(t)dt\,d_\infty(u, v)\,.
\end{equation}
\end{proposition}
\begin{proof}Let $u, v\in \mathcal{B}_T$.
Clearly, by Lemma \ref{lemma:psi-lemma-1-local}, $\Psi (u), \Psi(v)\in \mathcal{B}_T$.
    Consider now the following estimate for the quantity $|(\Psi u - \Psi v)(x, t)|$, for any $x\in G$, $t\in[0, T]$.
    \begin{equation}
        |(\Psi u - \Psi v)(x, t)|\leq \int_0^t\,\sum_{y\in X}\,p(x, y, t-s)\,h(s)|u(y,s)^q-v(y,s)^q|\mu(y)\,ds  .
        \label{eq310}
    \end{equation}
    Using the Lagrange theorem, the quantity $|u(y, s)^q-v(y, s)^q|$ can be controlled as follows. Let $y\in G$, $s\in(0, t)$; then there exists  $\xi$ in between  $u(y, s)$ and $v(y, s)$ such that
    \begin{equation*}
        |u(y, s)^q-v(y, s)^q| = q\,\xi^{q-1}|u(y, s)-v(y, s)|\,.
    \end{equation*}
    Since $u, v\in \mathcal{B}_T$, we have
    \begin{align*}
        |u(y, s)^q-v(y, s)^q|&\leq q\,\max\{u(y,s), v(y, s)\}^{q-1}|u(y, s)-v(y, s)|\\
        &\leq q\,M^{q-1} |u(y, s)-v(y, s)|\,.
    \end{align*}
    Combining the preceeding inequality  with \eqref{eq310} and \eqref{eq23}, we obtain
    \begin{align*}
        &|(\Psi u - \Psi v)(x, t)| \leq
        q\,M^{q-1}\int_0^t\,\sum_{y\in X}\,p(x, y, t-s)\, h(s)|u(y, s)-v(y, s)|\mu(y)\,ds \\
                &\leq d_\infty(u,v)\,q\,M^{q-1} \int_0^t\,\sum_{y\in X}\,p(x, y, t-s)\mu(y) h(s)\,ds \\
        &\leq d_\infty(u, v)\,q\,M^{q-1} \,\int_0^T h(s)\,ds \quad \text{ for all } x\in G, t\in (0,T).
    \end{align*}
    Thus \eqref{eq55aa} follows.
    \end{proof}

\begin{proof}[Proof of Theorem \ref{teo2}]

Let $M>\|u_0\|_\infty$. We can choose $T>0$ so small that \eqref{eq330} holds and
\begin{equation}\label{eq301}
q\,M^{q-1}\int_0^{T}\,h(t)dt<1\,.
\end{equation}
Therefore, in view of Proposition \ref{lemma:psi-lemma-2-local}, the map $\Psi:\mathcal{B}_T\rightarrow\mathcal{B}_T$ is a contraction map on a complete metric space. Hence by the Banach-Cacioppoli theorem, there exists a unique fixed point, that is a function $u\in\mathcal{B}_T$ satisfying
    \begin{equation*}
        u(x, t) = \sum_{y\in X}\,p(x, y, t)u_0(y)\mu(y) + \int_0^t\,\sum_{y\in X}\,p(x, y, t-s)\,h(s)u(y, s)^q\,ds\quad x\in G, t\in [0, T]\,.
    \end{equation*}
 Such $u$ is a mild solution of problem \eqref{problema}. Furthermore, since $u\in \mathcal B_T$, $u$ is bounded in $G\times (0, T).$
\end{proof}

\section{Global existence}\label{existence}
Also the proof of global existence relies on the contraction mapping principle. However, in this case we have to change properly the metric space in which the map is defined. It is introduced in the following definition and it is related to the heat kernel $p$.

\begin{definition}\label{funcspace}
Let $M>0, \gamma>0$. Let $y_0\in G$ be a point such that $u_0(y_0)>0$. We denote by $\mathcal{M}(M, \gamma, y_0)\equiv \mathcal M$ the set of all non-negative functions $u: G\times[0,+\infty)\to \mathbb R$ such that $t\mapsto u(x, t)$ is continuous for each $x\in G$, and \eqref{eq305} is fulfilled.
The distance $d:\mathcal{M}\times\mathcal{M}\rightarrow[0, +\infty)$ between any two functions $u, v\in\mathcal{M}$ is defined as
\begin{equation}\label{eq52}
d(u, v):=\sup_{x\in G, t>0}\frac{|(u-v)(x,t)|}{p(x,y_0,t+\gamma)}\,.
\end{equation}
\end{definition}

It can be easily proven that the space $(\mathcal{M}, d)$ is a complete metric space. We assume also \eqref{eq12a} for this section,
where $\epsilon$ is a suitable positive constant.
 We define the map $\Psi:\mathcal{M}\rightarrow\mathcal{M}$ as
\begin{equation}\label{eq:psi}
(\Psi u)(x,t):=\sum_{z\in G}p(x,z,t)u_0(z)\mu(z)+\int_0^t\sum_{z\in G}p(x,z,t-s)\,h(s)u(z,s)^q\,\mu(z)\,ds\,.
\end{equation}

To begin with, we show that $\Psi:\mathcal M\to \mathcal M$ is well-defined.

\begin{lemma}
    Suppose that  $u_0\in\ell^\infty(G)$ and that condition \eqref{eq12a} holds. Let $\delta>0$ and $M>0$  be such that
    \begin{equation}\label{eq302}
        \|u_0\|_{\ell^\infty(G)}<\delta\,,
    \end{equation}
    and
    \begin{equation}
        M\leq\frac{\delta}{\underline{C}}\,,
        \label{eq:lemma-psi-1-M}
    \end{equation}
    where $\underline C$ is given by \eqref{e301}. Suppose that \eqref{eq306} is satisfied; moreover,
        \begin{equation}
        \delta^{q-1}\,\tilde{H}<1\,,
        \label{eq:lemma-psi-1-delta}
    \end{equation}
    and
    \begin{equation}
        0<\epsilon<M\,(1-\delta^{q-1}\,\tilde{H})\,.
        \label{eq:lemma-psi-1-epsilon}
    \end{equation}
    Then, for any $v\in\mathcal{M}$, we have
    \begin{equation*}
        \Psi(v)\in\mathcal{M\,.}
    \end{equation*}
    \label{lemma:psi-lemma-1}
\end{lemma}

\begin{proof}
In view of \eqref{eq:psi}, we have, for any $t\geq 0$, $x\in G$,
\begin{equation}\label{eq311}
\begin{aligned}
        0\leq (\Psi u)(x, t) &= \sum_{y\in X}\,p(x, y, t)u_0(y)\mu(y)\,+ \\
        &+\int_0^t\,\sum_{y\in X}\,p(x, y, t-s)\,\frac{u(y, s)}{p(y, y_0, s+\gamma)}\,p(y, y_0, s+\gamma)\,h(s)u(y, s)^{q-1}\mu(y)\,ds  \\
        &\leq \epsilon\,\sum_{y\in X}\,p(x, y, t)\,p(y, y_0, \gamma)\mu(y)  \\
        &+M\,\int_0^t\,\sum_{y\in X}\,p(x, y, t-s)\,p(y, y_0, s+\gamma)\,h(s)u(y, s)^{q-1}\mu(y)\,ds\,,
    \end{aligned}
\end{equation}
    By using the fact that the function $u\in\mathcal{M}$, the estimate \eqref{e301} for the heat kernel, and the hypothesis \eqref{eq:lemma-psi-1-M} on the constant $M$,
    the quantity $u(y, s)^{q-1}$ inside the integral can be estimated as follows
    \begin{equation}\label{dim:parte-2}
        u(y, s)^{q-1}\leq M^{q-1}\,p(y, y_0, s+\gamma)^{q-1}\leq M^{q-1}\,\underline{C}^{q-1}\,e^{-\lambda_1(G)(q-1)t}
        \leq \delta^{q-1} e^{-\lambda_1(G)(q-1)s}\,.
        \end{equation}
    Hence due to \eqref{eq311}, \eqref{dim:parte-2} and \eqref{eq24}, we obtain, for all $x\in G, t>0,$
    \begin{equation*}
        (\Psi u)(x, t) \leq  \epsilon\,p(x, y_0, t+\gamma)+M\,\delta^{q-1}\,p(x, y_0, t+\gamma)\,\int_0^t\,e^{-\lambda_1(G)(q-1)s}h(s)\,ds\,.
    \end{equation*}
    Hence
    \begin{equation*}
        (\Psi u)(x, t) \leq (\epsilon+M\delta^{q-1}{\widetilde H})\,p(x, y_0, t+\gamma) \quad \text{ for all } x\in G, t\geq 0\,.
    \end{equation*}
    Thanks to hypotheses \eqref{eq306} and \eqref{eq:lemma-psi-1-delta}, it is possible to choose the constant $\epsilon$ so that
    \eqref{eq:lemma-psi-1-epsilon} is fulfilled. Therefore, we can infer that $\Psi(u)\in\mathcal{M}$.
\end{proof}

We are left to prove that the map $\Psi$ is in fact a contraction in the space $(\mathcal{M}, d)$.

\begin{proposition}\label{lemma:psi-lemma-2}
Let the assumptions of Theorem \ref{teo2} and Lemma \ref{lemma:psi-lemma-1} be satisfied.
Then, for any $u,v\in\mathcal{M}$,
\begin{equation}\label{eq55}
d(\Psi(u), \Psi(v))\leq\,q\,\delta^{q-1}\,\tilde{H}\,d(u, v)\,.
\end{equation}
In particular, if \begin{equation}\label{eq303}
        q\,\delta^{q-1}\,\tilde{H}<1\,,
    \end{equation}
    then $\Psi:\mathcal{M}\rightarrow\mathcal{M}$ is a contraction map.
\end{proposition}
\begin{proof}
For any $x\in G$, $t\geq 0$, we have that
    \begin{equation}
        |(\Psi u - \Psi v)(x, t)|\leq \int_0^t\,\sum_{y\in X}\,p(x, y, t-s)\,h(s)|u(y,s)^q-v(y,s)^q|\mu(y)\,ds
        \label{eq:stima-per-contraction-local}
    \end{equation}
    We now control the quantity $|u(y, s)^q-v(y, s)^q|$. Let $y\in G$, $s\in(0, t)$; by the Lagrange theorem there exists $\xi$ in between $u(y, s)$ and $v(y, s)]$ such that
    \begin{equation*}
        |u(y, s)^q-v(y, s)^q| = q\,\xi^{q-1}|u(y, s)-v(y, s)|\,.
    \end{equation*}
   Since $u, v\in \mathcal{M}$, by means of the estimate \eqref{e301} for the heat kernel, we can deduce that
    \begin{align*}
        |u(y, s)^q-v(y, s)^q|&\leq q\,\max\{u(y,s), v(y, s)\}^{q-1}|u(y, s)-v(y, s)|\\
        &\leq q\,M^{q-1}p(x, y_0, s+\gamma)^{q-1}|u(y, s)-v(y, s)|\\
        &\leq q\,M^{q-1}\underline{C}^{q-1}e^{-\lambda_1(q-1)s}|u(y, s)-v(y, s)|\\
        &\leq q\,\delta^{q-1}e^{-\lambda_1(q-1)s}|u(y, s)-v(y, s)|\,.
    \end{align*}
    Combining this estimate with (\ref{eq:stima-per-contraction-local}), and exploiting \eqref{eq24}, we obtain
    \begin{align*}
        &|(\Psi u - \Psi v)(x, t)| \leq
        q\,\delta^{q-1}\int_0^t\,\sum_{y\in X}\,p(x, y, t-s)\,e^{-\lambda_1(q-1)s}h(s)|u(y, s)-v(y, s)|\mu(y)\,ds \\
        &= q\,\delta^{q-1}\int_0^t\,\sum_{y\in X}\,p(x, y, t-s)\,e^{-\lambda_1(q-1)s}h(s)\frac{|u(y, s)-v(y, s)|}{p(y, y_0, s+\gamma)}p(y, y_0, s+\gamma)\mu(y)\,ds \\
        &\leq d(u,v)\,q\,\delta^{q-1}\int_0^t\,\sum_{y\in X}\,p(x, y, t-s)\,p(y, y_0, s+\gamma)\mu(y)\,e^{-\lambda_1(q-1)s}h(s)\,ds \\
        &\leq d(u, v)\,q\,\delta^{q-1}\,p(x, y_0, t+\gamma)\,\int_0^t\,e^{-\lambda_1(q-1)s}h(s)\,ds\,.
    \end{align*}
    Dividing the inequality by $p(x, y_0, t+\gamma)$, we get
    \begin{equation*}
        d(\Psi(u), \Psi(v))\leq d(u, v)\,q\,\delta^{q-1}\int_0^{t}\,e^{-\lambda_0(q-1)s}h(s)\,ds.
    \end{equation*}
Thus
    \begin{equation*}
        \frac{|(\Psi u - \Psi v)(x, t)|}{p(x, y_0, t+\gamma)}\leq q\,\delta^{q-1}\,\widetilde{H}\,d(u, v)\,,
    \end{equation*}
    Therefore, by taking the supremum over $G\times(0, +\infty)$, we have
    \begin{equation*}
        d(\Psi(u), \Psi(v))\leq q\,\delta^{q-1}\,\widetilde{H}\,d(u, v)\,.
    \end{equation*}
\end{proof}

\begin{proof}[Proof of Theorem \ref{teo2}]
 In view of the hypotheses made on $u_0$ and $\tilde H$, we can assume that \eqref{eq302}-\eqref{eq:lemma-psi-1-epsilon} and \eqref{eq305} are fulfilled.
    Therefore the map $\Psi:\mathcal{M}\rightarrow\mathcal{M}$ is a contraction in a complete metric space. Hence by the Banach-Caccioppoli theorem, there exists a unique fixed point, that is a function $u\in\mathcal{M}$ satisfying
    \begin{equation*}
        u(x, t) = \sum_{y\in X}\,p(x, y, t)u_0(y)\mu(y) + \int_0^t\,\sum_{y\in X}\,p(x, y, t-s)\,h(s)u(y, s)^q\,ds\quad x\in X, t\geq 0\,,
    \end{equation*}
 Such $u$ is a mild solution of problem \eqref{problema}. Furthermore, since $u\in \mathcal M$, we can infer that \eqref{eq305} is satisfied.
\end{proof}

{\bf Data availability statement}. There are no data associated with this research.

%
%

%


\end{document}